\documentclass[11pt]{amsart}
\usepackage{epsfig}
\usepackage{graphicx}
\usepackage{amscd}
\usepackage{amsmath}
\usepackage{amsxtra}
\usepackage{amsfonts}
\usepackage{amssymb}

\newtheorem{theorem}{Theorem}[section]

\newtheorem{proposition}[theorem]{Proposition}

\theoremstyle{definition}
\newtheorem{definition}[theorem]{Definition}
\newtheorem{remark}[theorem]{Remark}

\newtheorem{example}[theorem]{Example}
\theoremstyle{remark}

\renewcommand{\theclaim}{\textup{\theclaim}}

\newtheorem*{acknowledgements}{Acknowledgements}

\numberwithin{equation}{section}

\def\openone

{\mathchoice

{\hbox{\upshape \small1\kern-3.3pt\normalsize1}}

{\hbox{\upshape \small1\kern-3.3pt\normalsize1}}

{\hbox{\upshape \tiny1\kern-2.3pt\SMALL1}}

{\hbox{\upshape \Tiny1\kern-2pt\tiny1}}}

\makeatletter

\newbox\ipbox

\newcommand{\ip}[2]{\left\langle #1\, , \,#2\right\rangle}
\newcommand{\diracb}[1]{\left\langle #1\mathrel{\mathchoice

{\setbox\ipbox=\hbox{$\displaystyle \left\langle\mathstrut
#1\right.$}

\vrule height\ht\ipbox width0.25pt depth\dp\ipbox}

{\setbox\ipbox=\hbox{$\textstyle \left\langle\mathstrut
#1\right.$}

\vrule height\ht\ipbox width0.25pt depth\dp\ipbox}

{\setbox\ipbox=\hbox{$\scriptstyle \left\langle\mathstrut
#1\right.$}

\vrule height\ht\ipbox width0.25pt depth\dp\ipbox}

{\setbox\ipbox=\hbox{$\scriptscriptstyle \left\langle\mathstrut
#1\right.$}

\vrule height\ht\ipbox width0.25pt depth\dp\ipbox}

}\right. }

\newcommand{\dirack}[1]{\left. \mathrel{\mathchoice

{\setbox\ipbox=\hbox{$\displaystyle \left.\mathstrut
#1\right\rangle$}

\vrule height\ht\ipbox width0.25pt depth\dp\ipbox}

{\setbox\ipbox=\hbox{$\textstyle \left.\mathstrut
#1\right\rangle$}

\vrule height\ht\ipbox width0.25pt depth\dp\ipbox}

{\setbox\ipbox=\hbox{$\scriptstyle \left.\mathstrut
#1\right\rangle$}

\vrule height\ht\ipbox width0.25pt depth\dp\ipbox}

{\setbox\ipbox=\hbox{$\scriptscriptstyle \left.\mathstrut
#1\right\rangle$}

\vrule height\ht\ipbox width0.25pt depth\dp\ipbox}

} #1\right\rangle}

\newcommand{\cj}[1]{\overline{#1}}

\newcommand{\bz}{\mathbb{Z}}

\newcommand{\br}{\mathbb{R}}

\newcommand{\bt}{\mathbb{T}}
\newcommand{\bn}{\mathbb{N}}

\def\blfootnote{\xdef\@thefnmark{}\@footnotetext}

\input xy
\xyoption{all}
\usepackage{amssymb}

\def\F{\mathcal{F}}

\def\-{^{-1}}

\begin{document}

\title[Isospectral measures]{Isospectral measures}
\author{Dorin Ervin Dutkay}
\blfootnote{}
\address{[Dorin Ervin Dutkay] University of Central Florida\\
	Department of Mathematics\\
	4000 Central Florida Blvd.\\
	P.O. Box 161364\\
	Orlando, FL 32816-1364\\
U.S.A.\\} \email{ddutkay@mail.ucf.edu}

\author{Palle E.T. Jorgensen}
\address{[Palle E.T. Jorgensen]University of Iowa\\
Department of Mathematics\\
14 MacLean Hall\\
Iowa City, IA 52242-1419\\}\email{jorgen@math.uiowa.edu}

\thanks{} 
\subjclass[2000]{}

\begin{abstract}
In recent papers a number of authors have considered Borel probability measures $\mu$ in $\br^d$ such that the Hilbert space $L^2(\mu)$ has a Fourier basis (orthogonal) of complex exponentials. If $\mu$ satisfies this property, the set of frequencies in this set is called a spectrum for $\mu$. Here we fix a spectrum, say $\Gamma$, and we study the possibilities for measures $\mu$ having $\Gamma$ as spectrum.  
\end{abstract}
\maketitle \tableofcontents

\section{Introduction}
  We consider a spectral analysis of families of singular measures, introducing pairs $(\mu, \Gamma)$ where $\mu$ is a measure, and where $\Gamma$ is a set which serves as  spectrum for $\mu$; see the definition below. We refer to these as spectral pairs. While the measures arising this way have a special flavor, they are nonetheless useful in the analysis of models arising in a host of different areas.

We are motivated in part by a renewed interest in families of singular measures, driven in  turn both by  applications, and by current problems in spectral theory and  geometric measure theory. The applications include Schroedinger operators from physics, especially their scattering theory \cite{An09, Ab08, FLZ08}.  In these problems, it is helpful to have at hand concrete model-examples involving measures amenable to direct computations.  In stochastic processes and stochastic integration, key tools depend on underlying spectral densities. For problems involving fluctuations and chaotic dynamics, the measures are often singular, and model-measures are helpful, see e.g.,  \cite{Lu10, BoPo09, GiVl09, LaRa09, AlLe10, AlBe09}. In determining the nature of orbits in ergodic theory, the first question is often ``what is the spectral type?'' The measures in these applications are typically not compactly supported. Nonetheless, there is a procedure from geometric measure theory which produces compactly supported measures, see \cite{Hut81}, and much of the earlier literature has focused on measures of compact support. Our results below show that non-compactly supported measures arise in every spectral pair.

We consider Borel probability measures $\mu$ in $\br^d$ such that the Hilbert space $L^2(\mu)$ has a Fourier basis (orthogonal, ONB for short) of complex exponentials; for $\lambda\in\br^d$, we use the notation
$$e_\lambda(x):=e^{2\pi i\lambda\cdot x},\quad(x\in\br^d).$$ If $\mu$ satisfies this property, the set of frequencies in an ONB are called a {\it spectrum} for $\mu$. Here we fix a spectrum, say $\Gamma$, and we study the possibilities for measures $\mu$ having $\Gamma$ as spectrum.
The reader will notice that there are examples with Lebesgue measure restricted to the unit interval based on translation congruence. 
But it is of interest to extend these constructions to the general case (including cases with non-compact support but spectral). The more traditional spectral pairs $(\mu, \Gamma)$ typically come from contractive affine iterated function systems, and these measures $\mu$ will automatically be compactly supported in $\br^d$.
 
Here we explore the general theory, and we find very general and varied families: a rich family of iso-spectral fractals. There is a number of reasons that it is of interest to explore such; details below!
 
To help the reader understand the ideas, consider Borel probability measures that are iso-spectral; i.e., every measure in the family has the same spectrum, say $\Gamma$. Specifically, we fix a spectrum. What are natural families of Borel probability measures with this spectrum?
 
Indeed, consider the most general spectral pair $(\mu, \Gamma)$ in $\br^d$; and then develop algorithms yielding indexed families of measures $(\mu_a)$ with the index $a$ in a specific index set $A$. A natural choice for $A$ is a suitable family of partitions of a basepoint measure $\mu$ in the family.  We find extensive families, but there are probably other bigger and intriguing families of iso-spectral measures. 
 
We elaborate on the set $A$ below in Theorem \ref{th1.2}. The two steps in the algorithm consist in choosing measurable partitions of the $d$-cube (in $\br^d$) and translations by $\bz^d$ defining a translation congruence; see Definition \ref{def1.3}. The set $A$ of all translation congruences labels the iso-spectral measures. 
 
In the simplest case, take  $a=$ the trivial partition; and this yields back $\mu$ itself, up to a translation.
There are at least four reasons this is of interest: 
 \begin{itemize}
 \item[(a)] There is a big classical literature going back to Mark Kac's question: ``Can one hear the shape of a drum?'' 
In our context, we will be considering the Fourier transform of $\mu$ as a function on $\br^d$, and the possible drums will be iso-spectral data, typically iso-spectral fractals, or fractal measures. The idea of making connections between geometric features and spectral theoretic data, of course dates back to Fourier, but it was made popular by Mark Kac in \cite{Ka66}. While Kac had in mind a Laplacian on a planar domain, the question has generated a host of formulations involving various forms of spectral data, and various geometries; see for example \cite{La08} and the references cited there. Here we are concerned with Fourier frequencies on one side of the divide, and geometric measure theory on the other. 
 
\item[(b)] The Fourier transform of $\mu_a$, $F_a(t)= \widehat\mu_a(t)$, $t\in\br^d$ is interesting as we vary $a\in A$. We can get $F_a(t)$ non-differentiable; and anything in between continuous and entire analytic. In a more fundamental setting, the problem of recovering a function of a measure from a Fourier transform lies at the root of obstacle scattering, but in a classical context, the paper \cite{HD63} illustrates some of the subtleties.
\item[(c)] And going back to Paley-Wiener there is much literature on the interplay between the possibility of analytic continuations of geometric Fourier transforms and the geometry itself. Here, by Paley-Wiener, we mean questions dealing with asymptotic estimates on a complex Fourier transform. The classical literature includes \cite{PZ09} and papers cited there, but, by contrast, there are relatively few parallel results dealing with classes of fractal measures.
 \item[(d)] There is an analogy of these families to wavelet sets that play an important role in the spectral theory of wavelets in higher dimensions. Here we refer to tiling properties for wavelet sets, see e.g., \cite{Me08}. This is only a parallel as wavelet sets involve two operations, translation and scaling. Our focus here is on translations by integer lattices.
 \end{itemize}

We further discuss applications as they relate to tiling properties of the integer lattices $\bz^d$, both for $d = 1$, as well as for higher dimensions. The examples we include are from the theory of wavelets and from harmonic analysis of affine fractals.

\section{Main results}
While it is not true in general, for given measures $\mu$,  that the Hilbert space  $L^2(\mu)$ is amenable to Fourier analysis, at least not in a direct way, in earlier papers, e.g., \cite{DHJ09, DuJo07a, JoPe98} the notion of spectrum was introduced for the analysis of certain singular measures $\mu$. Here our aim is to fix sets $\Gamma$ that  serve as spectrum, and ask for the variety of measures $\mu$ that have $\Gamma$ as their spectrum. In fact, in the setting from \cite{DuJo09a} our Theorem \ref{th1.1} offers a rather complete answer to this question.

 Our study of spectral pairs $(\mu, \Gamma)$ extends the more familiar theory of Pontryagin duality for locally compact abelian groups. 
For other results on spectral pairs, the reader may wish to consult  \cite{JoPe96, JoPe98, DuJo09a, DuJo09b, DuJo07a, DuJo07b, DuJo08a, DuJo08b}. For recent related results see also \cite{JKS07} and \cite{JKS08}.
The simplest instance of interesting spectral pairs include the compact $d$-torus  $\bt^d$ and its Fourier dual the rank-$d$ lattice  $\bz^d$,  the setting of multivariable Fourier series. In this context, the required and more standard Fourier tools for $d=1$ do carry over to $d > 1$. Instead one may rely on the canonical Fourier duality of the $d$-torus $\bt^d = \br^d/\bz^d$    and $\bz^d$; see \cite{Rud90}. It will be convenient to model $\bt^d$ as the $d$-cube in $\br^d$, i.e., as $Q_d := I^d$, where $I := [0, 1)$.
 
      There are many differences between classical multivariable Fourier analysis (e.g., \cite{Rud90}) on the one hand, and spectral pairs $(\mu, \Gamma)$ on the other; for example this: the absence of groups in the context of general spectral pairs. Indeed, typically for general spectral pairs, neither of the two sets in the pair, the support of the measure nor its spectrum, is a group.

     Nonetheless, there are important spectral theoretic questions for those particular spectral pairs where the measure $\mu$ is $d$-dimensional Lebesgue measure restricted to $Q_d$. There are several questions here; first: What sets $\Gamma$ in $\br^d$ make $(Q_d, \Gamma)$ a spectral pair? The answer was found in the two papers \cite{JoPe99, IoPe98}. These authors proved that a discrete subset $\Gamma$ in $\br^d$ is a spectrum for $Q_d$ if and only if it tiles $\br^d$ by translations of $Q_d$.
      
This spectral/tile duality in fact is a part of a wider duality theory. It was initiated in \cite{JoPe99}, and further studied in the follow-up paper \cite{IoPe98}. The first one of the two papers introduced the problem, solved it and listed the possibilities for $d<4$. The second proved the theorem in general, but without classification. Higher dimensions are of interest because of the existence of exotic ``cube-tilings'' in $\br^d$ for $d = 10$ and higher, see \cite{LaSh92}, found by Lagarias and Shor.
      Our purpose here is to turn the question around: Rather than fixing one part in a particular spectral pair, in this case $Q_d$, instead we pick the simplest spectrum $\Gamma = \bz^d$, and we then ask what are the possibilities for measures $\mu$ in spectral pairs $(\mu, \bz^d)$. The theorem below answers this question!

\begin{theorem}\label{th1.1}
Let $\mu$ be a Borel probability measure on $\br^d$. The following statements are equivalent:
\begin{enumerate}
\item The set $\{e_n : n\in\bz^d\}$ forms an orthonormal set in $L^2(\mu)$.
\item There exists a bounded measurable function $\varphi\geq 0$ that satisfies
\begin{equation}
\sum_{k\in\bz^d}\varphi(x+k)=1,\mbox{ for Lebesgue a.e. }x\in\br^d,
\label{eq1.1}
\end{equation}
such that $d\mu=\varphi\,dx$. 
\end{enumerate}
\end{theorem}

\begin{proof}
Let $Q:=[0,1)^d$.
Assume (ii). Then, for $n\in\bz^d$
$$\int_{\br^d}e_n(x)\,d\mu(x)=\int_{\br^d}e_n(x)\varphi(x)\,dx=\sum_{k\in\bz^d}\int_Q\varphi(x+k)e_n(x+k)\,dx=$$$$\int_Qe_n(x)\left(\sum_{k\in\bz^d}\varphi(x+k)\right)\,dx=\int_Qe_n(x)\,dx=\left\{\begin{array}{cc}
1,&\mbox{ if }n=0,\\
0,&\mbox{ if }n\neq 0.\end{array}\right.$$
This proves (i).

Assume (i). Define the measure $\tilde\mu$ on $Q$ as follows: for all Borel subsets $E$ of $Q$:
\begin{equation}
\tilde\mu(E)=\sum_{k\in\bz^d}\mu(E+k).
\label{eq1.2}
\end{equation}
Approximating uniformly any continuous functions on $[0,1)^d$ by step functions we obtain 
\begin{equation}
\int_Q f\,d\tilde\mu=\sum_{k\in\bz^d}\int_{Q+k}f(x-k)\,d\mu(x),\mbox{ for all continuous functions $f$ on $Q$.}
\label{eq1.3}
\end{equation}

Then 
$$\int_Q e_n\,d\tilde\mu=\sum_{k\in\bz^d}\int_{Q+k} e_n(x-k)\,d\mu(x)=\sum_{k\in\bz^d}\int_{Q+k}e_n(x)\,d\mu=\int_{\br^d}e_n(x)\,d\mu(x)=\delta_n,$$
by assumption. 

But then, using the Stone-Weierstrass theorem, it follows that $\tilde\mu$ must coincide with the Lebesgue measure on $Q$ on continuous functions, so they are the same.

Next, we prove that $\mu$ is absolutely continuous w.r.t the Lebesgue measure on $\br^d$. Suppose there is a Borel set $E$ of Lebesgue measure zero with $\mu(E)>0$. Then for some $k\in\bz^d$, $\mu(E\cap(Q+k))>0$. Then define $F:=(E\cap(Q+k)))-k \subset Q$. We will have that $F$ has Lebesgue measure zero, but $\tilde\mu(F)>\mu(F+k)>0$, which contradicts the fact that $\tilde\mu$ is Lebesgue measure on $Q$. 

Thus $\mu$ is absolutely continuous w.r.t. Lebesgue measure $\lambda$ on $\br^d$. Let $\varphi$ be the Radon-Nykodim derivative $d\mu/dx$. 
We have 
$$\lambda(E)=\tilde\mu(E)=\sum_{k\in\bz^d}\mu(E+k)=\sum_{k\in\bz^d}\int_{Q+k}\varphi(x)\chi_E(x-k)\,dx=\sum_{k\in\bz^d}\int_Q\varphi(x+k)\chi_E(x)\,d\mu(x)=$$$$
\int_Q\chi_E(x)\left(\sum_{k\in\bz^d}\varphi(x+k)\right)\,dx.$$
This implies \eqref{eq1.1} and (ii).

\end{proof}

\begin{remark}
There are simple examples of measures for which $\bz$ yields orthogonal set of exponentials but which are not complete.

Consider the measure $\mu=\frac12\lambda|_{[0,2)}$. Note that $\varphi:=\frac12\chi_{[0,2)}$ satisfies (ii) in Theorem \ref{th1.1}. So $\{e_n : n\in\bz\}$ is orthonormal in $L^2(\mu)$. Actually more is true, and this can be seen by a simple rescale by 2: the measure $\mu$ has spectrum $\frac12\bz$, so in particular $\{e_n : n\in\bz\}$ is an orthonormal set in $L^2(\mu)$, but incomplete. The function $\varphi$ in this case is $\varphi=\frac12\chi_{[0,2)}$. The question is: what extra conditions have to be imposed on $\varphi$ so that $\{e_n : n\in\bz\}$ is complete?
\end{remark}

\begin{definition}\label{def1.3}
We say that a Borel subset $E$ of $\br^d$ is {\it translation congruent} to $Q=[0,1)^d$ if there exists a measurable partition $\{E_k : k\in\bz^d\}$ of $Q$ such that
$$E=\bigcup_{k\in\bz^d}(E_k+k).$$
\end{definition}

Our theme is this:  fix a set $\Gamma$ arising as a spectrum. We ask for the variety of measures $\mu$ that have $\Gamma$ as their spectrum. The result below answers the question for the special case when $\Gamma = \bz^d$. The relevance of this setting is illustrated by the problems in the papers \cite{JoPe99, LaSh92}. The question in \cite{JoPe99}, inspired in part by \cite{LaSh92}, deals with the possibility of spectral pairs when one term in the pair is the $d$-cube $Q$ in $\br^d$. The classification of the spectra was found for small $d$. The authors of \cite{JoPe99} further suggested that spectra have the additional property that they are translation sets for additive tilings. Our result (Theorem \ref{th1.2}) about translation congruence offers a possible answer to this. 
\begin{theorem}\label{th1.2}
Let $\mu$ be a Borel probability measure on $\br^d$. Then $\mu$ has spectrum $\bz^d$ iff $\mu$ is the Lebesgue measure restricted to a set $E$ which is translation congruent to $Q$. 
\end{theorem}

\begin{proof}
Suppose $\mu$ is Lebesgue measure restricted to a set $E$ which is translation congruent to $Q$ and let $E_k$ be as in Definition \ref{def1.3}. Define the measurable map $\psi:Q\rightarrow E$, $\psi(x)=x+k$ for all $x\in E_k$ and $k\in\bz^d$. It is easy to see that the operator $\Psi: L^2(E)\rightarrow L^2(Q)$, $\Psi f=f\circ\psi$ is an isometric isomorphism and $\Psi(e_n|_E)=e_n|_{Q}$. So $\bz^d$ is a spectrum for $\mu$.

Conversely, suppose $\bz^d$ is a spectrum for $\mu$. Let $\varphi$ be as in Theorem \ref{th1.1}(ii). 

Define 
\begin{equation}
F_\varphi(t,x):=\sum_{k\in\bz^d}e_k(t)\varphi(k+x),\quad(t,x\in\br^d).
\label{eq2.4}
\end{equation}

The sum is absolutely convergent, from \eqref{eq1.1}. 

Note that for $n\in\bz^d$,
$$F_\varphi(t,x+n)=\sum_ke_k(t)\varphi(k+n+x)=\sum_{m}e_{m-n}(t)\varphi(m+x)=e_{-n}(t)F_\varphi(t,x).$$

We have 
$$c_\varphi(t):=\sum_{n\in\bz^d}|\widehat\varphi(t+n)|^2=\sum_{n\in\bz^d}\widehat{\varphi^*\ast\varphi}(t+n)=$$
$$=\sum_{k\in\bz^d}e_k(t)(\varphi^*\ast\varphi)(k)\quad (\mbox{by Poisson's summation formula, since }\varphi\in L^1(\br^d))$$
$$=\sum_{k\in\bz^d}e_k(t)\int_{\br^d}\cj{\varphi(y)}\varphi(y+k)\,dy=\int_{\br^d}\varphi(y)F_\varphi(t,y)\,dy=\sum_{n\in\bz^d}\int_Q\varphi(x+n)F_\varphi(t,x+n)\,dx$$
$$=\sum_{n\in\bz^d}\int_Q\varphi(x+n)e_{-n}(t)F_\varphi(t,x)\,dx=\int_Q\cj{F_\varphi(t,x)}F_\varphi(t,x)\,dx=\int_Q|F_\varphi(t,x)|^2\,dx.$$

To see that the interchange of the sum and the integral is justified, note that by \eqref{eq1.1}, we have $|F_\varphi|\leq 1$ and the sum that defines it is absolutely convergent. Also $\int_{\br^d}|\varphi(y)|\,dy=1$ (since $\varphi\geq 0$ and from \eqref{eq1.1}). 

Note that by the definition \eqref{eq2.4}, we get the following formula
\begin{equation}\label{eq2.1}
|F_\varphi(t,x)|=\left|\sum_ke_k(t)\varphi(x+k)\right|\leq \sum_{k\in\bz^d}\varphi(x+k)=1.
\end{equation}
So $$
c_\varphi(t)\leq\int_Q|F_\varphi(t,x)|^2\,dx\leq 1.
$$

We know (see \cite{JoPe98,DuJo07b}), that $\bz^d$ is a spectrum for $\mu$ iff $c_\varphi(t)=1$ for all $t\in\br^d$. But this implies that for all $t\in\br^d$, $|F_\varphi(t,x)|=1$ for Lebesgue a.e. $x$. 

Take $t\in\br^d$. Take $x\in\br^d$. Since $\sum_l\varphi(x+k)=1$, there exists $k_0$ such that $\varphi(x+k_0)\neq 0$. Since we must have equality in 
the triangle inequality in \eqref{eq2.1}, all the terms in the sum must differ by a real multiplicative constant; i.e., for all $k$ such that $\varphi(x+k)\neq 0$, we must have $e_k(t)\varphi(x+k)=\alpha_ke_{k_0}(t)\varphi(x+k_0)$ for some real $\alpha_k$. This implies that $e_k(t)=e_{k_0}(t)$. So $(k-k_0)\cdot t\in\bz$, and, picking several $t$'s with irrational components, it follows that $k$ must be equal to $k_0$. 

Thus, for a.e. $x$, there is a unique $k_0$ depending on $x$, such that $\varphi(x+k_0)\neq 0$. Since $\sum_k\varphi(x+k)=1$, it follows that $\varphi(x+k_0)=1$. This implies the desired conclusion.
\end{proof}

\begin{example}($d=1$)\label{ex1.4}
Let 
$$\varphi=\frac23\chi_{[0,1)}+\frac13\chi_{[1,2)}.$$
Let $d\mu=\varphi\,dx$. By Theorem \ref{th1.1}, the set $\{e_n : n\in\bz\}$ is orthogonal in $L^2(\mu)$. 
We have 
$$\widehat\mu(t)=\frac{1}{6\pi i t}\left(e^{2\pi i 2t}+e^{2\pi i t}-2\right),\quad(t\in\br).$$
This shows that $\widehat\mu(t)=0$ iff $t\in\bz$. From this we see that there is no $t\in\br\setminus\bz$ such that $e_t$ is orthogonal to $e_n$ for $n\in\bz$, because $\ip{e_t}{e_n}_{L^2(\mu)}=\widehat\mu(t-n)$.

Therefore $\bz$ yields a {\it maximal} set of orthogonal exponentials, which is incomplete (by Theorem \ref{th1.2}). 
\end{example}

Next, we give consider a different case, we characterize the measures that have spectrum $\{0,\dots,N-1\}$ for some finite integer $N$. The simplest example is of course $\frac{1}N\sum_{k=0}^{N-1}\delta_{1/k}$, where $\delta_x$ is the Dirac measure at $x$.

\begin{theorem}\label{th2.6}
Let $N\geq 2$ be an integer. Let $A$ be a set in $\br$ such that the atomic measure $\delta_A=\frac{1}{N}\sum_{a\in A}\delta_a$ has spectrum $\{0,1,\dots,N-1\}$. The $A$ is of the form $A=\frac{1}{N}A'$ where $A'$ is a compete set of representatives for $\bz/N\bz$. 
\end{theorem}

\begin{proof}
It is easy to see, by writing the orthogonality of the exponential functions, that $A$ has spectrum $\{0,\dots,N-1\}$ iff the matrix 
$$\frac{1}{\sqrt{N}}\left(e^{2\pi i ak}\right)_{a\in A,k\in\{0,\dots,N-1\}}$$
is unitary. 

If $A$ has the form given in the statement of the theorem, then this matrix is unitary; it is the matrix of the Fourier transform on the group $\bz/N\bz$. 

For the converse, assume the matrix is unitary. Then for any pair of distinct points $a,a'$ in $A$, we must have
$$\sum_{k=0}^{N-1}e^{2\pi i(a-a')k}=0$$
so $e^{2\pi i(a-a')}$ is a root of the polynomial $\sum_{k=0}^{N-1}z^k$. Then $a-a'=\frac{l}{N}$ for some $l\in\bz$, not a multiple of $N$. Since there are $N$ elements in $A$, the pigeon hole principle implies that $NA$ is a complete set of representatives for $\bz/N\bz$.
\end{proof}

\begin{definition}\label{def2.6}
Given a Borel measure $\mu$ on $\br^d$, a family of Borel subsets $(E_i)_{i\in I}$ is called a {\it partition of $\mu$} if $\mu(\br^d\setminus\cup_i E_i)=0$ and $\mu(E_i\cap E_j)=0$ for all $i\neq j$. 
We say that two Borel measures $\mu$ and $\mu'$ are {\it translation equivalent} if there exists a partition $(E_i)_{i\in I}$ and some integers $(k_i)_{i\in I}$ of $\mu$ such that $(E_i+k_i)_{i\in I}$ is a partition of $\mu'$, and
the functions $E_i\ni x\mapsto x+k_i\in E_i+k_i$ map the measure $\mu$ into the measure $\mu'$.
\end{definition}

\begin{proposition}\label{pr2.7}
Let $\mu$ and $\mu'$ be two translation equivalent Borel probabiliy measures on $\br^d$. If $\mu$ has a spectrum $\Gamma$ contained in $\bz^d$, then $\mu'$ is also a spectral measure with spectrum $\Gamma$.
\end{proposition}

\begin{proof}
Let $(E_i)_{k\in\bz^d}$ be a partition of $\mu$ and $(k_i)_{i\in I}$ as in Definition \ref{def2.6}. Define the map $\psi:\br^d\rightarrow\br^d$, $\psi(x)=x+k_i$ for $x\in E_i$. Then it is easy to check that the map 
$\Psi:L^2(\mu')\rightarrow L^2(\mu)$, $\Psi(f)=f\circ\psi$, is an isometric isomorphism with the property that 
$\Psi(e_\lambda)=e_\lambda$ for all $\lambda\in\bz^d$. Since $\Gamma$ is contained in $\bz^d$, it follows that it is also a spectrum for $\mu'$. 
\end{proof}

\begin{remark}\label{rem2.8}
Theorems \ref{th1.2} and \ref{th2.6} might lead one think that if two measures $\mu$ and $\mu'$ have a common spectrum $\Gamma$ contained in $\bz$, then they must be translation equivalent. However, this is not true, as the following example shows. 

Consider the atomic measures $\delta_A$ and $\delta_{A'}$, where $A=\left\{0,\frac18,\frac48,\frac58\right\}$ and $A'=\left\{0,\frac38,\frac48,\frac78\right\}$. They have the common spectrum $\Gamma=\{0,1,4,5\}$. This can be seen by computing the matrices, $\frac{1}{\sqrt{4}}(e^{2\pi i a\lambda})_{a\in A,\lambda\in\Gamma}$ and similarly for $A'$, with $\rho=e^{2\pi i/8}$:
$$\frac{1}{\sqrt{4}}\begin{pmatrix}
	1&1&1&1\\
	1&\rho&-1&-\rho\\
	1&-1&1&-1\\
	1&-\rho&-1&\rho
\end{pmatrix},\quad\frac{1}{\sqrt{4}}\begin{pmatrix}
	1&1&1&1\\
	1&\rho^3&-1&-\rho^3\\
	1&-1&1&-1\\
	1&-\rho^3&-1&\rho^3
\end{pmatrix}$$
which are unitary. However, the measures $\delta_A$ and $\delta_{A'}$ are not translation equivalent.
\end{remark}

\begin{example}\label{ex2.9}
Let $\mu$ be the invariant measure associated to the affine iterated function system $\tau_0(x)=x/4$, $\tau_2(x)=(x+2)/4$. See \cite{JoPe98,DJ06b} for definition and details. It is proved is \cite{JoPe98} that this measure has spectrum 
$$\Gamma=\left\{\sum_{k=0}^n4^k a_k : a_k\in \{0,1\}, n\in\bn\right\}.$$
Consider also the invariant measure $\mu'$ for the affine iterated function system $\tau_0(x)=x/4$, $\tau_{10}(x)=(x+10)/4$. With the notations in \cite{DJ06b}, $R=4$ and $B=\{0,10\}$. Then, with $L:=\{0,1\}$, we have a Hadamard pair $(B,L)$. We want to show that $\mu'$ has the same spectrum $\Gamma$. For this we use \cite[Theorem 8.4]{DJ06b} and check that there are no extreme cycles (or $m_B$-cycles as they are called in \cite{DJ06b}) for this system. These are finite sets of points $\{x_0,x_1,\dots,x_{p-1}\}$ such that $|m_B(x_i)|=1$ for all $i$, where 
$$m_B(x)=\frac{1}{2}(1+e^{2\pi i 10x}),$$
and for all $i$ there exist $l_i\in L$ such that $\frac14(x_i+l_i)=x_{i+1}$, where $x_{p}:=x_0$. 

Such a cycle will be contained in the attractor $X(L)$ of the affine iterated function system $\sigma_0(x)=x/4$, $\sigma_1(x)=(x+1)/4$, so it will contained in the interval $[0,1/3]$. 
Also, since $|m_B(x_i)| =1$, we have that $x_i=k/10$ for some $k\in\bz$. Therefore the only non-zero candidates are $1/10,2/10,3/10$. It can be checked that none of them is an $m_B$-cycle. 
Then \cite[Theorem 8.4]{DJ06b} implies that $\Gamma$ is a spectrum for $\mu'$. 

We show that $\mu'$ is not translation congruent to $\mu$. 

Note that $\mu'$ is supported inside $[0,\sum_{k=1}^\infty10/4^k]=[0,10/3]$. Also $10/3$ is a fixed point for $\tau_{10}$, so there is a small interval around $10/3$, for example $\tau_{10}^n[0,10/3]$, that has positive $\mu'$ measure. 
$10/3$ is congruent to $1/3$ modulo $\bz$. However the measure $\mu$ is supported inside the interval $[0,2/3]$, and after one iteration, we can see that it is also supported inside $[0,2/12]\cup[6/12,8/12]$. But $1/3=4/12$ is at positive distance from this union, and therefore no piece in the support of $\mu$ can cover the piece of $\mu'$ around $10/3$.

\end{example}
\begin{proposition}\label{pr2.8}
Let $\mu$ be a Borel probability measure with spectrum $\Gamma$ contained in $\bz^d$. Then $\mu$ is translation equivalent with a measure supported on $Q:=[0,1)^d$. 

\end{proposition}

\begin{proof}
First, we prove that we cannot have a Borel set $O$ such that $O$ and $O+k$ are contained in the support of $\mu$, for some $k\in\bz^d$, $k\neq 0$ and $\mu(O)>0$. 
Assume by contradiction that this is the case. Then all the hypotheses in \cite[Proposition 2.2]{DHJ09} are satisfied, so $\mu$ is locally translation invariant, and therefore we have $\mu(O+k)=\mu(O)>0$ and, for all $\lambda\in\bz$
$$\int e_\lambda\chi_O\,d\mu=\int e_\lambda(x-k)\chi_O(x-k)\,d\mu(x)=\int e_\lambda \chi_{O+k}\,d\mu.$$
This implies that $e_\lambda\perp (\chi_O-\chi_{O+k})$. But since $\Gamma$ is a spectrum, this yields a contradiction. 

Thus, if $O$ is contained in the support of $\mu$ and has positive measure then $O+k$ must have measure zero, if $k\in\bz^d$, $k\neq 0$. 

We can then partition the measure $\mu$ into $E_k=\mbox{supp}(\mu)\cap (Q+k)$, and define the measure $\tilde\mu$ as in \eqref{eq1.2}. The property above implies that $\tilde\mu$ is translation equivalent to $\mu$ and of course it is supported on $Q$. 
\end{proof}

\begin{definition}\label{defft}\cite{JoPe99}. 
Let $\mu$ be a spectral probability measure on $\br^n$, with spectrum $\Gamma$ a subset of $\br^n$.

Define the Fourier transform $\F:L^2(\mu)\rightarrow l^2(\Gamma)$ by 
$$(\F f)(\lambda)=\ip{f}{e_\lambda},\quad(f\in L^2(\mu),\lambda\in\Gamma).$$
Then $\F$ is unitary and  
$$\F^{-1}(c_\lambda)_\lambda=\sum_{\lambda\in\Gamma}c_\lambda e_\lambda.$$

Define the group of transformations $(U(t))_{t\in\br^n}$ on $L^2(\mu)$ by
$$U(t)f=\F^{-1}((e_t(\lambda)\F f(\lambda))_\lambda)=\sum_{\lambda} [e^{2\pi i t\cdot\lambda}\ip{f}{e_\lambda}] e_\lambda.$$
The convergence of the sum is in $L^2(\mu)$.
This means that in the ``Fourier domain'', $\hat U(t):=\F U(t) \F^{-1}$ is just multiplication by the sequence $(e^{2\pi i t\cdot\lambda})_\lambda$. We call $(U(t))_{t\in\br^n}$ the {\it group of local translations}.

Note also that 
\begin{equation}
	U(t)e_\lambda=e_\lambda(t)e_\lambda,\quad(t\in \br^n,\lambda\in\Gamma).
	\label{eqeig}
\end{equation}
Note that $U(t)$ depends on the spectrum $\Gamma$. 
\end{definition}

The next proposition shows that, for two iso-spectral measures, the natural map that interchanges the exponential bases will also intertwine the groups of local translations. 

\begin{proposition}\label{pr2.10}
Let $\mu$ and $\mu'$ be two spectral measures having the same spectrum $\Gamma$. Define the map $\Psi:L^2(\mu)\rightarrow L^2(\mu')$, $\Psi(e_\lambda)=e_\lambda$, for all $\lambda\in\Gamma$. (Note that, in general, $\Psi$ is not the identity map, it just maps the restriction of the exponential function $e_\lambda$ to the support of $\mu$ to the restriction of $e_\lambda$ to the support of $\mu'$). Let $\F_\mu$ and $\F_{\mu'}$ be the Fourier transforms and let $(U_\mu(t))_{t\in\br^d}$ and $(U_{\mu'}(t))_{t\in\br^d}$ be the corresponding groups of local translations. 

Then $\Psi$ is an isometric isomorphism, $\Psi=\F_{\mu'}^{-1}\F_\mu$ and $\Psi$ intertwines $U_\mu$ and $U_{\mu'}$, i.e.
\begin{equation}
\Psi U_\mu(t)=U_{\mu'}(t)\Psi,\quad(t\in\br^d)
\label{eq2.7}
\end{equation}
\end{proposition}

\begin{proof}

 We have $\F_\mu e_\lambda=\delta_\lambda=\F_{\mu'}e_\lambda$. Since $\{e_\lambda : \lambda\in\Gamma\}$ is an ONB, it follows that $\Psi=\F_{\mu'}^{-1}\F_\mu$. The rest follows by a simple computation.

\end{proof}

\medskip
{\bf Conclusion.} Our main result (Theorem \ref{th1.2}) is a characterization of probability measures $\mu$  supported in $\br^d$ which allow $L^2(\mu)$-orthogonal Fourier series indexed by $\bz^d$. We argue how our result fits into a wider context of making links between geometric shapes, on one side, and spectral data on the other. Here we are concerned with Fourier frequencies on one side of the divide, and geometric measure theory on the other.
\begin{acknowledgements}
The authors thank an anonymous referee for very helpful suggestions, much improving the paper, both in form and substance.
\end{acknowledgements}
\bibliographystyle{alpha}
\bibliography{2d}

\end{document}